 \documentclass[authoryear,final,3p,times]{elsarticle}

 \usepackage{etex}

\usepackage{amssymb}
\usepackage{amsmath}
\usepackage{amsfonts}
\usepackage{mathtools}
\usepackage{bm}
\usepackage{tikz}
\usepackage{dsfont}
\usetikzlibrary{trees,shapes}
\usepackage{pifont}
\usepackage{pgfplots}
\usepackage{subcaption}
\usepackage[all]{xy}

\usepackage{amsthm}

\newdefinition{example}{Example}
\newtheorem{theorem}{Theorem}

\newtheorem{proposition}{Proposition}
\newdefinition{definition}{Definition}
\newtheorem{corollary}{Corollary}

\usepackage{verbatim}
\usepackage{array,amsmath,booktabs}
\usepackage{sansmath}
\usepackage{hyperref}
\usepackage{natbib}
\pgfplotsset{compat=newest}
 
 \usepackage{bbm}

\begin{document}

\begin{frontmatter}


\title{Sensitivity analysis beyond linearity}


\author{Manuele Leonelli}

\address{School of Mathematics and Statistics, University of Glasgow, UK.}

\begin{abstract}
A wide array of graphical models can be parametrised to have atomic probabilities represented by monomial functions. Such monomial structure has proven very useful when studying robustness under the assumption of a multilinear model where all monomial have either zero or one exponents. Robustness in probabilistic graphical models is usually investigated by varying some of the input probabilities and observing the effects of these on output probabilities of interest.  Here the assumption of multilinearity is relaxed and a general approach for sensitivity analysis in non-multilinear models is presented. It is shown that in non-multilinear models sensitivity functions have a polynomial form, conversely to multilinear models where these are simply linear. The form of various divergences and distances under different covariation schemes is also formally derived. Proportional covariation is proven to be optimal in non-multilinear models under some specific choices of varied parameters.  The methodology is illustrated throughout by an educational application.
\end{abstract}


\end{frontmatter}


\section{Introduction}
\label{sec:introduction}
Sensitivity methods have received great attention in the literature of probabilistic graphical models in the past twenty years. Sensitivity analysis is a fundamental part of any applied analysis, carried out to validate the construction of a probabilistic graphical model and investigate its robustness to misspecification of its probabilities. Such methods have been successfully used in a variety of applications \citep[e.g.][]{Nur2009,Oberguggenberger2009,Pollino2007,Uusitalo2007}.

Research has mostly focused on Bayesian network (BN) models \citep{Koller2009,Smith2010}, although sensitivity results also exist for Markov networks \citep{Chan2005b} and chain event graphs \citep{Leonelli2017}. Sensitivity analysis in BNs usually consists of two phases: first some parameters of the model are varied and the effect of these variations on output probabilities of interest are investigated; second, once parameter variations are identified, the effect of these are summarized by a distance or divergence measure between the original and the varied distributions underlying the BN. Although sensitivity methods exist for continuous random variables under the assumption of Gaussianity \citep[e.g][]{Castillo2003, Gomez2013, Gorgen2018}, henceforth we focus on the most common case of discrete random variables only.

For the first phase of a sensitivity analysis, a simple mathematical function, usually termed \textit{sensitivity function}, describes an output probability of interest as a function of the BN parameters. This is a (multi-) linear function of the varied parameters for marginal output probabilities \citep{Castillo1997,Coupe2002}. Conversely, if the probability of interest is a conditional probability, then the sensitivity function is a ratio of (multi-) linear functions. 

For the second phase, the Chan-Darwiche distance \citep{Chan2005}, Kullback-Leibler divergence \citep{Kullback1951} and $\phi$-divergences \citep{Ali1966} are often used to measure the overall effect of parameter variations. One important line of research has focused on identifying parameter \textit{covariations}, i.e. ways to adjust parameters so to respect the sum to one condition after a parameter variation, that minimize such distances. Proportional covariation \citep{Laskey1995,Renooij2014}, which assigns the same proportion of residual probability mass to covarying paramaters after a variation, is the gold-standard method since this has been shown to minimize the above-mentioned divergences in a variety of settings \citep{Chan2002,Leonelli2017}, although not all \citep{Leonelli2018}.
 
Most of the above-mentioned results, although specifically derived for BNs, hold for a variety of models whose atomic probabilities can be written as a multilinear polynomial \citep{Leonelli2017}. The multilinear structure of atomic probabilities in BNs has been known for quite some time \citep{Castillo1995,Darwiche2003}, but other models entertain the same property under specific parametrisations, for instance stratified staged trees \citep{Gorgen2015}, context-specific BNs \citep{Boutilier1996} and influence diagrams \citep{Leonelli2017a}.

The development of sensitivity methods for models whose atomic probabilities cannot be written as multilinear polynomials have been limited. Results have been derived for dynamic Bayesian networks (DBNs) \citep{Charitos2006,Charitos2006a}, Markov chains \citep{Cooman2008} and hidden Markov models \citep{Amsalu2017,Renooij2012}. The atomic probabilities of all these model classes have a non-square-free polynomial representation, as demonstrated in \citet{Brandherm2004}
since they all have a DBN characterisation. Non-multilinear atomic probabilities are often associated to models whose probabilities are recursively updated through time in a dynamic fashion, although this does not necessarily have to be the case as demonstrated by the examples below.

This work presents a general framework for sensitivity analysis in models whose atomic probabilities have a non-multilinear structure and therefore can be applied to the already mentioned model classes of DBNs and hidden Markov models. The monomial representation of a statistical model introduced in \citet{Leonelli2018} is used here to encompass all classes of discrete models with non-multilinear atomic probabilities. For such models, the form of the sensitivity functions and their properties are derived. Furthermore, results about the computation of the CD distance and $\phi$-divergences under various covariation schemes are derived. In particular, it is proven that, for specific choices of parameters to be varied, proportional covariation is optimal, in the sense that it minimizes the CD distance between the original and varied distributions amongst all possible ways to covary parameters. Therefore, this work extends the results of \citet{Leonelli2017} for multilinear models to non-multilinear ones, as well as proposing sensitivity methods similar to those of \citet{Renooij2012} and \citet{Charitos2006} but which apply to a much  more general class of models.

The paper is structured as follows. Section \ref{sec:monomial} reviews monomial models and shows that staged trees have in general a non-multilinear polynomial representation. This section further introduces a running example from an educational application. Section \ref{sec:cov} reviews covariation methods for probabilities. Section \ref{sec:sensi} reports the derivations of the sensitivity functions for non-multilinear models, whilst Section \ref{sec:div} deals with divergences and their computation. The paper is concluded with a discussion.

\section{Monomial models}
\label{sec:monomial}
A review of monomial models, in short MMs, as introduced in \citet{Leonelli2018} is given first. Let $\mathbb Y$ be a finite set with $q$ elements  and 
 $\operatorname{P}$  a strictly positive probability density function for $\mathbb Y$.  
Let $\#\mathbb{Y}=q$, call $ y\in \mathbb Y$ an atom and $\operatorname{P}(y)$ the atomic probability of $ y$.  
The generic probability $\operatorname{P}$ can be seen as a point in the interior set of the $q$-dimensional simplex, i.e. $\operatorname{P}\in \Delta_{q-1}$. 
Next, a particular class of parametric statistical models, called MMs, is associated to $\mathbb Y$.

Let $[k]=\{1,2,\ldots,k\}$.
A MM is defined by three elements:  a $q\times k$ matrix $A$ with non-negative integer entries, $A\in\mathcal{M}_{q\times k}(\mathbb{Z}_{\geq 0})$;  a $k$-dimensional parameter vector $\theta$ with positive real entries,
$\theta=(\theta_i)_{i\in[k]}\in\mathbb{R}^k_{>0}$; and a partition $S=\{S_1,\dots,S_n\}$ of $[k]$. 
There is a row of $A$ for each atom $y$ and $A_y$ indicates the $y$-th row of $A$. 
The atomic probability of $y\in\mathbb{Y}$ given $\theta$ and $A$ is defined as  $\operatorname{P}( y ) = \prod_{i \in [k]}  \theta_i^{ A_{ y,i}}=\theta^{A_y}$. The partition $S$ of $[k]$ is such that $\theta_{S_i}=(\theta_j)_{j\in S_i}\in \Delta_{\# S_i-1}$. The atomic probability of $y\in\mathbb{Y}$ can then be written as 
\[
\operatorname{P}(y)=\prod_{i\in[n]}\prod_{j\in S_i}\theta_j^{A_{y,j}}=\prod_{i\in[n]}\theta_{S_i}^{A_{y,S_i}},
\]
where $\theta_S^{A_{y,S}}=\prod_{i\in S}\theta_i^{A_{y,i}}$ denotes the monomial associated to an event $y\in\mathbb{Y}$ where only parameters $\theta_i$ for $i\in S$ can have non-zero exponent.  For $A\in\mathcal{M}_{q\times k}(\mathbb{Z}_{\geq 0})$, $B\subseteq [q]$ and $C\subseteq [k]$,  $A_{B,C}$ denotes the submatrix of $A$ with $B$ rows and $C$ columns.

\begin{definition}
\label{def:MM}
 The MM  over $\mathbb Y$ associated to $A$, $\theta$ and $S$, where $S$ is such that $\theta_{S_i}\in\Delta_{\# S_i-1}$, is  defined as 
\begin{align*}
\operatorname{MM}(A,\theta, S) =
\left\{
\operatorname{P}  \in \Delta_{q-1} : \operatorname{P}(y) = \prod_{i\in[n]}\theta_{S_i}^{A_{y,S_i}}  \mbox{ for }y\in\mathbb{Y}  \mbox{ and } \theta\in\mathbb{R}^k_{>0}
\right\} 
\end{align*} 
 A $\operatorname{MM}(A,\theta,S)$ is said to be multilinear if $A\in\mathcal{M}_{q\times k}(\{0,1\})$.
\end{definition} 
A MM is multilinear if all its monomials are square free, i.e. the exponents of the parameters are either zero or one. \citet{Leonelli2017} and \citet{Leonelli2018} give a thorough investigation of sensitivity analysis in multilinear MMs. Here conversely the focus is on models which are not necessarily multilinear.

\begin{example}
\label{ex:1}
Consider a simple coin toss game. The probability of head (H) is $\theta_1$, whilst tail (T) has probability $\theta_2$, where $\theta_1+\theta_2=1$. If the result of the first toss is head, then the coin is tossed a second time. This situation can be represented by a MM with parameter vector $\theta=(\theta_1,\theta_2)$, degenerate partition of $[2]$ including one element only, and matrix $A$ defined as
\[
A=\begin{pmatrix}
2&0\\
1&1\\
0&1\\
\end{pmatrix}
\]
where the first column of $A$ relates to $\theta_1$ and the second to $\theta_2$. The model is such that $\operatorname{P}(HH)=\theta_1^2$, $\operatorname{P}(HT)=\theta_1\theta_2$ and $\operatorname{P}(T)=\theta_2$. This MM is non-multilinear since the matrix A includes an entry equal to 2.
\end{example}

Since DBNs have been already shown to have a non-multilinear monomial structure in \citet{Brandherm2004}, here the focus is on staged trees, which are introduced next.

\subsection{Staged trees}
Graphical models represented by \textit{event trees} $\mathcal{T}=(V,E)$  are considered here, which are directed rooted trees where each inner vertex $v\in V$ has at least two children. In this context, the sample space of the model corresponds to the set of root-to-leaf paths in the graph and each directed path, which is a sequence of edges $r=(e~ |~ e\in E(r))$, for $E(r)\subset E$ has a meaning in the modelling context. Each edge $e\in E$ is associated to a primitive probability $\theta_e\in(0,1)$ such that on each \textit{floret} $\mathcal{F}(v)=(v,E(v))$, where $E(v)\subseteq E$ is the set of edges emanating from $v\in V$, the primitive probabilities sum to unity. The probability of an atom is then simply the product of the primitive probabilities along the edges of its path: $\operatorname{P}(r)=\prod_{e\in E(r)}\theta_e$.

\begin{definition}
Let $\theta_v=(\theta_e~|~ e\in E(v))$ be the vector of primitive probabilities associated to the floret $\mathcal{F}(v)$, $v\in V$, in an event tree $\mathcal{T}=(V,E)$. A \emph{staged tree} is an event tree as above where, for some $v,w\in V$, the floret probabilities are identified $\theta_v=\theta_w$. Then, $v,w\in V$ are in the same \emph{stage}.
\end{definition}

Two vertices are thus in the same stage if they have the same (conditional) distribution over their edges. When drawing a tree, vertices in the same stage are either framed using the same shape or equally colored in order to have a visual counterpart of that information. Setting floret probabilities equal can be thought of as representing conditional independence information. Staged trees are capable of representing all conditional independence hypotheses within discrete BNs, whilst at the same time being more flexible in expressing modifications of these \citep{Gorgen2015,Smith2008}.

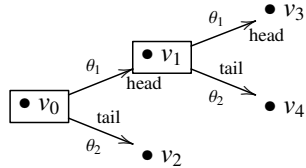
\begin{figure}[t]
\[ \xymatrixrowsep{0.5pc}{{\xymatrix{
&&\bullet~v_3\\
&*+[F]{\bullet~v_1}\ar[ru]^{\theta_1}_<(0.7){\text{head}}\ar[rd]_{\theta_2}^{\text{tail}}\\
*+[F]{\bullet~v_0}\ar[ru]^{\theta_1}_<(0.7){\text{head}}\ar[rd]_{\theta_2}^{\text{tail}}&&\bullet~v_4\\
&\bullet~v_2
}}}\]\vspace*{-10pt}
\caption[]{The staged tree of a repeated coin toss from Example \ref{ex:1}.}\label{fig:treecoin}
\end{figure}

Staged trees are MMs whose atomic probabilities can either be multilinear or not \citep{Gorgen2015}. The following example gives a simple illustration of a non-multilinear staged tree.

\begin{example}
The MM of Example \ref{ex:1} can be depicted as the staged tree in Figure \ref{fig:treecoin}, which has two inner-vertices, $v_0$ and $v_1$, in the same stage. The tree has three root-to-leaf paths ending in the leaves $v_3$ (head and head), $v_4$ (head and tail) and $v_2$ (tail). The edges emanating from the inner-vertices $v_0$ and $v_1$ are associated to the primitive probabilities $\theta_1$ and $\theta_2$ representing the probability of head and tail respectively.
\end{example}

\subsection{An example}
\label{sec:23}
To illustrate the construction of a staged tree and its monomial representation, an example from an educational application is considered. This example was first introduced in \citet{Freeman2011}.

In a one-year program students take components A and B, but not everyone in the same order: students are first allocated to study either module A or B for the first six months and then the other for the final six months. After the first six months students are examined on their allocated component and can be awarded a distinction (D), a pass (P) or a fail (F). If failed, they can resit the exam with the possibility of passing and thus be allowed to the second component. Students who fail the resit are withdrawn from the program. For the second module students can again either fail, pass or be awarded a distinction, but with no possibility of resitting. With an obvious extension of the labeling, the process can be depicted by the tree in Figure \ref{fig:tree2}

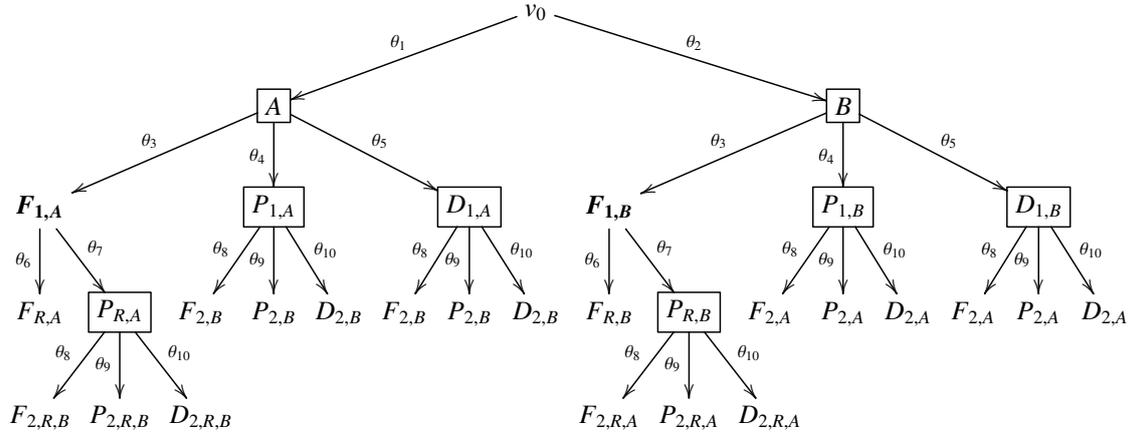
\begin{figure}[t]
\[ \xymatrixcolsep{0.1pc}{{\xymatrix{
&&&&&&&v_0\ar[rrrrd]^{\theta_2}\ar[lllld]_{\theta_1}\\
&&&*+[F]{A}\ar[llld]_{\theta_3}\ar[d]_{\theta_4}\ar[rrrd]^{\theta_5}&&&&&&&&*+[F]{B}\ar[llld]_{\theta_3}\ar[d]_{\theta_4}\ar[rrrd]^{\theta_5}\\
\bm{F_{1,A}}\ar[dr]^{\theta_7}\ar[d]_{\theta_6}&&&*+[F]{P_{1,A}}\ar[dl]_{\theta_8}\ar[d]_{\theta_9}\ar[dr]^{\theta_{10}}&&&*+[F]{D_{1,A}}\ar[dl]_{\theta_8}\ar[d]_{\theta_9}\ar[dr]^{\theta_{10}}&&\bm{F_{1,B}}\ar[d]_{\theta_6}\ar[dr]^{\theta_7}&&&*+[F]{P_{1,B}}\ar[ld]_{\theta_8}\ar[d]_{\theta_9}\ar[rd]^{\theta_{10}}&&&*+[F]{D_{1,B}}\ar[ld]_{\theta_8}\ar[d]_{\theta_9}\ar[rd]^{\theta_{10}}\\
F_{R,A}&*+[F]{P_{R,A}}\ar[ld]_{\theta_8}\ar[d]_{\theta_9}\ar[rd]^{\theta_{10}}&F_{2,B}&P_{2,B}&D_{2,B}&F_{2,B}&P_{2,B}&D_{2,B}&F_{R,B}&*+[F]{P_{R,B}}\ar[ld]_{\theta_8}\ar[d]_{\theta_9}\ar[rd]^{\theta_{10}}&F_{2,A}&P_{2,A}&D_{2,A}&F_{2,A}&P_{2,A}&D_{2,A}\\
F_{2,R,B}&P_{2,R,B}&D_{2,R,B}&&&&&&F_{2,R,A}&P_{2,R,A}&D_{2,R,A}
}}}\]\vspace*{-10pt}
\caption[]{The staged tree of the educational application of Section \ref{sec:23} under the first set of hypotheses.}\label{fig:tree2}
\end{figure}

Various hypotheses of conditional independence, corresponding to equal primitive probabilities of multiple florets, can be embedded in the above educational scenario. One set of such hypotheses was given in \citet{Freeman2011} as:
\begin{itemize}
\item The components A and B are equally hard: this corresponds to an equal framing of the vertices A and B in Figure \ref{fig:tree2}.
\item The chances of passing the first module after a fail do not depend on the module taken: this is depicted by an equal colouring of $F_{1,A}$ and $F_{1,B}$ in Figure \ref{fig:tree2}.
\item The distribution of grades for the last six months does not depend on the module taken nor on the results of the first part: this is depicted by framing $P_{R,A}$, $P_{1,A}$, $D_{1,A}$, $P_{R,B}$, $P_{1,B}$ and $D_{1,B}$ by a rectangle in Figure \ref{fig:tree2}. 
\end{itemize}

These hypotheses give the staged tree of Figure \ref{fig:tree2}, which can be equally represented by a MM with parameter vector $(\theta_1,\dots,\theta_{10})$, matrix $A=(A_{11},A_{12})^{\textnormal{T}}$, with
\[
A_{11}=
\begin{pmatrix}
1&0&1&0&0&1&0&0&0&0\\
1&0&1&0&0&0&1&1&0&0\\
1&0&1&0&0&0&1&0&1&0\\
1&0&1&0&0&0&1&0&0&1\\
1&0&0&1&0&0&0&1&0&0\\
1&0&0&1&0&0&0&0&1&0\\
1&0&0&1&0&0&0&0&0&1\\
1&0&0&0&1&0&0&1&0&0\\
1&0&0&0&1&0&0&0&1&0\\
1&0&0&0&1&0&0&0&0&1
\end{pmatrix}\hspace{1cm}
A_{12}=
\begin{pmatrix}
0&1&1&0&0&1&0&0&0&0\\
0&1&1&0&0&0&1&1&0&0\\
0&1&1&0&0&0&1&0&1&0\\
0&1&1&0&0&0&1&0&0&1\\
0&1&0&1&0&0&0&1&0&0\\
0&1&0&1&0&0&0&0&1&0\\
0&1&0&1&0&0&0&0&0&1\\
0&1&0&0&1&0&0&1&0&0\\
0&1&0&0&1&0&0&0&1&0\\
0&1&0&0&1&0&0&0&0&1
\end{pmatrix}
\]
and partition $S=\{S_1,S_2,S_3,S_4\}$ where $S_1=\{1,2\}$, $S_2=\{3,4,5\}$, $S_{3}=\{6,7\}$ and $S_4=\{8,9,10\}$. This model is multilinear since all entries of $A$ are either zero or one. Graphically this could have also been deduced by noticing that no vertices along a root-to-leaf path are in the same stage. 

\begin{figure}[t]
\[ \xymatrixcolsep{0.1pc}{{\xymatrix{
&&&&&&&v_0\ar[rrrrd]^{\theta_2}\ar[lllld]_{\theta_1}\\
&&&\bm{A}\ar[llld]_{\theta_3}\ar[d]_{\theta_4}\ar[rrrd]^{\theta_5}&&&&&&&&\bm{B}\ar[llld]_{\theta_3}\ar[d]_{\theta_4}\ar[rrrd]^{\theta_5}\\
*+[F]{F_{1,A}}\ar[dr]^{\theta_7}\ar[d]_{\theta_6}&&&\bm{P_{1,A}}\ar[dl]_{\theta_3}\ar[d]_{\theta_4}\ar[dr]^{\theta_5}&&&\bm{D_{1,A}}\ar[dl]_{\theta_3}\ar[d]_{\theta_4}\ar[dr]^{\theta_5}&&*+[F]{F_{1,B}}\ar[d]_{\theta_6}\ar[dr]^{\theta_7}&&&\bm{P_{1,B}}\ar[ld]_{\theta_3}\ar[d]_{\theta_4}\ar[rd]^{\theta_5}&&&\bm{D_{1,B}}\ar[ld]_{\theta_3}\ar[d]_{\theta_4}\ar[rd]^{\theta_5}\\
F_{R,A}&\bm{P_{R,A}}\ar[ld]_{\theta_3}\ar[d]_{\theta_4}\ar[rd]^{\theta_5}&F_{2,B}&P_{2,B}&D_{2,B}&F_{2,B}&P_{2,B}&D_{2,B}&F_{R,B}&\bm{P_{R,B}}\ar[ld]_{\theta_3}\ar[d]_{\theta_4}\ar[rd]^{\theta_5}&F_{2,A}&P_{2,A}&D_{2,A}&F_{2,A}&P_{2,A}&D_{2,A}\\
F_{2,R,B}&P_{2,R,B}&D_{2,R,B}&&&&&&F_{2,R,A}&P_{2,R,A}&D_{2,R,A}
}}}\]\vspace*{-10pt}
\caption[]{The staged tree of the educational application of Section \ref{sec:23} under the second set of hypotheses.}\label{fig:tree3}
\end{figure}
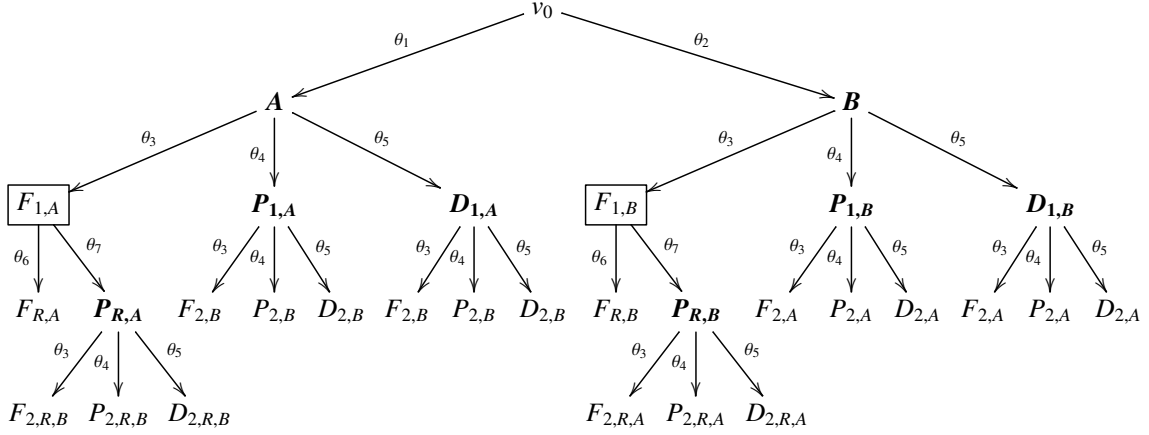

A second set of hypotheses may embellish the first one by assuming that the distribution of grades of students not experiencing fails are the same in all components. This additional hypothesis gives the staged tree in Figure \ref{fig:tree3} where vertices $A$, $B$, $P_{R,A}$, $P_{1,A}$, $D_{1,A}$, $P_{R,B}$, $P_{1,B}$ and $D_{1,B}$ are now all in the same stage. This staged tree can be written as a MM with parameter $(\theta_{1},\dots, \theta_{7})$, matrix $A=(A_{21},A_{22})^{\textnormal{T}}$ with

\[
A_{21}=
\begin{pmatrix}
1&0&1&0&0&1&0\\
1&0&2&0&0&0&1\\
1&0&1&1&0&0&1\\
1&0&1&0&1&0&1\\
1&0&1&1&0&0&0\\
1&0&0&2&0&0&0\\
1&0&0&1&1&0&0\\
1&0&1&0&1&0&0\\
1&0&0&1&1&0&0\\
1&0&0&0&2&0&0
\end{pmatrix}\hspace{1cm}
A_{22}=
\begin{pmatrix}
0&1&1&0&0&1&0\\
0&1&2&0&0&0&1\\
0&1&1&1&0&0&1\\
0&1&1&0&1&0&1\\
0&1&1&1&0&0&0\\
0&1&0&2&0&0&0\\
0&1&0&1&1&0&0\\
0&1&1&0&1&0&0\\
0&1&0&1&1&0&0\\
0&1&0&0&2&0&0
\end{pmatrix}
\]
and partition $S=\{S_1,S_2,S_3\}$ where $S_1=\{1,2\}$, $S_2=\{3,4,5\}$ and $S_3=\{6,7\}$. Under this additional hypothesis the staged tree does not entertain a multilinear monomial parametrization, but only a non-multilinear one. For such models there is currently no established sensitivity theory to investigate their robustness.

\section{Covariation}
\label{sec:cov}
The basic underlying idea of sensitivity analysis is to vary some of the model's parameters and observe how such variations affect outputs of interest. However, when such variations are performed,  then some of the remaining parameters need to be adjusted (or to \emph{covary}) to respect the sum-to-one condition of probability measures. In the binary case when one of the two parameters is varied this is straightforward, since the second parameter will be equal to one minus the other. But in generic discrete finite cases there are multiple ways to covary parameters.

The theory of covariation from \citet{Renooij2014} is reviewed next, with a particular focus on its specific characterization for MMs given in \citet{Leonelli2018}. For a set $S$ and $i\in S$,  let $S^{-i}$ denote $S\setminus \{i\}$, $-S_j$ denote the set $[k]\setminus S_j$ and  let $|v|$ denote the sum of the elements of a vector $v$.

\begin{definition}
\label{def:covar}
Let $\theta$ be the parameter vector of a MM and $\theta_i$ be the parameter varied where $i\in S_j$. Let $\theta$ be 
partitioned as $\theta=(\theta_i,\theta_{S^{-i}_j},\theta_{-S_j})$  and let  $\tilde{\theta}_i\in(0,1)$. 
  A $\tilde\theta_i$-\emph{covariation scheme} is a function $\sigma:\bigtimes\limits_{k\in[n]}\Delta_{\#S_k-1}\rightarrow\bigtimes\limits_{k\in[n]}\Delta_{\#S_k-1}$ which fixes $\theta_i$ to $\tilde\theta_i$ and does not change $\theta_{-S_j}$, i.e.
\begin{align*}
\sigma: \hspace{0.5cm}  \bigtimes\limits_{i\in[n]}\Delta_{\#S_i-1}\longmapsto&\bigtimes\limits_{i\in[n]}\Delta_{\#S_i-1}\\
(\theta_i,\theta_{S^{- i}_j},\theta_{-S_j})\longmapsto &(\tilde\theta_i,\cdot,\theta_{-S_j}).
\end{align*}
\end{definition}

Thus $\theta_{S_j}$ denotes a vector of parameters that need to respect the sum to one condition,  $\tilde\theta_i$ denotes the new numerical specification of the parameter varied and $\theta_{-S_j}$ the parameter vector which is not affected by the variation. Consider as an example a staged tree model. In a staged tree the sets $S_k$, 
 $k \in [n]$, denote the conditional probability distributions of florets in different stages. Suppose one parameter from one stage is varied. Then the parameters associated to that same stage are covaried, whilst all others are held fixed. 

\begin{definition}
\label{def:schemes}
In the notation of Definition~\ref{def:covar}
\begin{itemize}
\item the $\tilde\theta_i$-\textit{proportional} covariation scheme $\sigma_{\operatorname{pro}}(\theta)=(\tilde\theta_i,\tilde\theta_{S^{- i}_j},\theta_{-S_j})$ is defined by setting 
	\[
		\tilde\theta_k=
		\frac{1-\tilde{\theta}_i}{1-\theta_i}\theta_k \qquad  \text{for all } k\in S_j^{- i}. 
	\] 
\item The $\tilde\theta_i$-\textit{uniform} covariation scheme, $\sigma_{\operatorname{uni}}(\theta)=(\tilde\theta_i,\tilde\theta_{S_j^{-i}},\theta_{-S_j})$ is defined by setting
	\[
	\tilde\theta_k=\frac{1-\tilde{\theta}_i}{\#S_j-1}  \qquad  \text{for all } k\in S^{-i}_j. 
	\]
\item The $\tilde\theta_i$-\textit{linear} covariation scheme  $\sigma_{\operatorname{lin}}(\theta)=(\tilde\theta_i,\tilde\theta_{S_j^{-i}},\theta_{-S_j})$ is defined by setting
	\[
	\tilde\theta_k=\gamma_k\tilde\theta_i+\delta_k  \qquad  \text{for all } k\in S^{-i}_j,
	\]
	where $\gamma_k$ and $\delta_k$ need to be chosen so that $\tilde\theta_i+|\tilde\theta_{S_j^{-i}}|=1$ 
	\end{itemize}
\end{definition}

Different covariation schemes may entertain different properties which, depending on the domain of application, might be more or less desirable~\citep[see][for a list]{Leonelli2017,Renooij2014}. Applying a linear covariation scheme is very natural: if for instance $\delta_k=-\gamma_k$, then $\tilde\theta_k=\delta_k(1-\tilde\theta_i)$ and the scheme assigns a proportion $\delta_k$ of the remaining probability mass to $\tilde\theta_k$. Notice that uniform and proportional schemes are specific instances of linear covariations. Another used covariation scheme  is the order-preserving one \citep[see][for details]{Renooij2014}.

\section{Sensitivity functions}
\label{sec:sensi}
Sensitivity functions represent the functional relationship between a parameter being varied and the output probability of an event of interest. These are often used in practice since, for instance, the parameter specifications of a MM may imply event probabilities which appears to be unreasonable to a user, although being a coherent consequence of his/her beliefs. Sensitivity functions depict the required change of a parameter that would give a reasonable event probability.

Consider a $MM(A,\theta,S)$ and an event $E\subset\mathbb{Y}$ of interest. Definition \ref{def:sensi} gives the probability of an event $E$ as a function of a covariation scheme.

\begin{definition}
\label{def:sensi}
Let $\sigma$ be a $\tilde\theta_i$-covariation scheme. For $\operatorname{P}\in MM(A,\theta,S)$, the probability $\sigma(\operatorname{P})(E)$ read as a function of $\tilde\theta_i$ is called sensitivity function associated to $\sigma$.
\end{definition}

The following theorem derives the general form of sensitivity functions in non-multilinear MMs as well as their form for specific covariation schemes. 

\begin{theorem}
\label{theo:1}
 Let $\operatorname{P}\in MM(A,\theta,S)$, $E\subset \mathbb{Y}$ and suppose the parameter $\theta_i$ is varied, where $i\in S_j$. Then
\begin{itemize}
\item for a generic $\theta_i$-covariation scheme $\sigma$
\begin{equation}
\label{eq:theo11}
\sigma(\operatorname{P})(E)= \sum_{y\in E}\tilde\theta_{S_j}^{A_{y,S_{j}}}\theta_{-S_j}^{A_{y,- S_j}}
\end{equation}
\item for proportional covariation $\sigma_{\operatorname{pro}}$
\begin{equation}
\label{eq:theo12}
\sigma_{\operatorname{pro}}(\operatorname{P})(E)=\sum_{y\in E}\tilde\theta_i^{A_{y,i}}\left(\frac{1-\tilde\theta_i}{1-\theta_i}\right)^{|A_{y,S_j^{-i}}|}\theta_{S_j^{-i}}^{A_{y,S_j^{-i}}}\theta_{-S_j}^{A_{y,-S_j}}
\end{equation}
\item for uniform covariation $\sigma_{\operatorname{uni}}$
\begin{equation}
\label{eq:theo13}
\sigma_{\operatorname{uni}}(\operatorname{P})(E)=\sum_{y\in E}\tilde\theta_i^{A_{y,i}}\left(\frac{1-\tilde\theta_i}{\# S_j-1}\right)^{|A_{y,S_j^{-i}}|}\theta_{-S_j}^{A_{y,-S_j}}
\end{equation}
\item for linear covariation $\sigma_{\operatorname{lin}}$
\begin{equation}
\label{eq:theo14}
\sigma_{\operatorname{lin}}(\operatorname{P})(E)=\sum_{y\in E}\tilde\theta_i^{A_{y,i}}\prod_{k\in S_j^{-i}}(\gamma_k\tilde\theta_i+\delta_k)^{A_{y,k}}\theta_{-S_j}^{A_{y,-S_j}}
\end{equation}
\end{itemize}  
\end{theorem}
\begin{proof}
For equation (\ref{eq:theo11}) notice that 
\[
\sigma(\operatorname{P})(E)=\sum_{y\in E}\tilde\theta^{A_y}=\sum_{y\in E}\tilde\theta_{S_j}^{A_{y,S_j}}\tilde\theta_{-S_j}^{A_{y,-S_j}}=\sum_{y\in E}\tilde\theta_{S_j}^{A_{y,S_j}}\theta_{-S_j}^{A_{y,-S_j}}.
\]
The form of the sensitivity function under different covariation schemes follows from equation (\ref{eq:theo11}) by plugging-in their definition given in Definition \ref{def:covar}. 
\end{proof}

From Theorem \ref{theo:1} is then easy to deduce the polynomial properties of the sensitivity function in general MMs.

\begin{corollary}
\label{cor:1}
For proportional, uniform and linear $\tilde\theta_i$-covariation schemes, the sensitivity function $\sigma(\operatorname{P})(E)$ is a polynomial in $\tilde\theta_i$ of degree $\max_{y\in E}|A_{y,S_j}|$.
\end{corollary}
This follows from the form of the sensitivity functions given in equation (\ref{eq:theo12})-(\ref{eq:theo14}).

Notice that differently to multilinear MMs, where the sensitivity function is linear for any linear covariation scheme, the sensitivity function is more generally polynomial in non-multilinear MMs. However, there are cases where sensitivity functions are simply linear, as formalized by the following corollary.

\begin{corollary}
\label{cor:2}
In the notation of Theorem \ref{theo:1}, if $0\leq |A_{y,S_j}|\leq 1$ for all $y\in E$, then $\sigma(\operatorname{P})(E)$ is a linear function of $\tilde\theta_i$ for any linear $\tilde\theta_i$-covariation scheme.
\end{corollary}
This follows from Corollary \ref{cor:1} since if $0\leq |A_{y,S_j}|\leq 1$ then the sensitivity function is a polynomial of degree 1.

The previous results formalize the form of sensitivity functions for marginal probabilities. Conditional sensitivity functions represent the functional relationship between conditional probabilities and a parameter varied.

\begin{corollary}
The conditional sensitivity function $\sigma(P)(E~|~C)$ is the ratio of sensitivity functions $\sigma(P)(E\cap C)/\sigma(P)(C)$, where each of these have the properties formalized in Theorem \ref{theo:1}, Corollary \ref{cor:1} and Corollary \ref{cor:2}.
\end{corollary}
This result easily follows from the definition of conditional probability.
\begin{table}
\begin{center}
{\renewcommand{\arraystretch}{1.5}
\begin{tabular}{|c|}
\hline
Multilinear staged tree \\
\hline
$\theta_1=0.5$, $\theta_2=0.5$, $\theta_3=0.2$, $\theta_4=0.7$, $\theta_5=0.1$, $\theta_6=0.35$, $\theta_7=0.65$, $\theta_8=0.1$, $\theta_9=0.5$, $\theta_{10}=0.4$\\
\hline
\hline
Non-multilinear staged tree\\
\hline
$\theta_1=0.5$, $\theta_2=0.5$, $\theta_3=0.15$, $\theta_4=0.6$, $\theta_5=0.25$, $\theta_6=0.35$, $\theta_7=0.65$\\
\hline
\end{tabular}}
\end{center}
\caption{Probability specifications for the staged trees in Section \ref{sec:23}.\label{table:1}}
\end{table}

\begin{example}
To illustrate the different form of sensitivity functions in multilinear and non-multilinear models, consider the staged trees from the educational example of Section \ref{sec:23}. The two staged tree structures are embellished by the probability specifications given in Table \ref{table:1}. For ease of comparison the probability distributions from the stages $\{v_0\}$ and $\{F_{1,A}, F_{1,B}\}$ are equally defined in the two trees. The  distribution of the stage $\{A,B,P_{1,A},D_{1,A},P_{1,B},D_{1,B},P_{R,A}, P_{R,B}\}$ in the non-multilinear staged tree of Figure \ref{fig:tree3} is such that the parameters $\theta_3$, $\theta_4$ and $\theta_5$ are chosen from the probabilities underlying the tree in Figure \ref{fig:tree2} as $(\theta_3+\theta_8)/2$, $(\theta_4+\theta_9)/2$  and $(\theta_5+\theta_{10})/2$, respectively. Suppose the parameter $\theta_4$ is varied in both cases: notice that for the first tree this is the probability of passing the exam in the second semester, whilst for the three in Figure \ref{fig:tree3} this is the probability of passing an exam at any point.

 The probabilities of four events are considered here. First, the sensitivity function for a $\theta_4$ variation   of not being admitted to the second semester is for both trees $\theta_1\theta_3\tilde\theta_6+\theta_2\theta_3\tilde\theta_6$, where $\tilde\theta_6$ depends on the covariation scheme used. Thus in both models this function is simply linear whenever the covariation scheme is linear, even though the second tree is a non-multilinear model. These sensitivity functions are reported in Figure \ref{fig:sensi1}. Under uniform covariation, the sensitivity function is the same for the two trees, whilst under proportional covariation they differ.
 
 The second event considered is failing the exam in the second semester. For the multilinear tree the associated sensitivity function can be written as $
 \theta_8(\theta_1+\theta_2)(\tilde\theta_3\theta_7+\tilde\theta_4+\tilde\theta_5)
$,
whilst for the non-multilinear tree this is $(\theta_1+\theta_2)(\tilde\theta_3^2\theta_7+\tilde\theta_4\tilde\theta_3+\tilde\theta_5\tilde\theta_3)$. Thus in this case the sensitivity function is a non-linear function of the varied parameter, as reported in Figure \ref{fig:sensi2}, but for both trees the sensitivity function is decreasing.

For the event of passing both exams with distinction the sensitivity functions for the two trees are highly different, as reported in Figure \ref{fig:sensi3}. For the multilinear tree, the sensitivity function is slightly increasing and almost identical for uniform and proportional covariation. Conversely, for the non-multilinear tree this is decreasing non-linearly. Formally, for the multilinear tree the sensitivity function is $(\theta_1+\theta_2)\tilde\theta_5\theta_{10}$, whilst for the non-multilinear tree this is $(\theta_1+\theta_2)\tilde\theta_5^2$.

Lastly, the conditional probability of obtaining a distinction in the first semester given that a distinction was given in the second one is computed. In this case, the sensitivity function is a ratio of polynomials and as such is not linear even for multilinear models. This is shown in Figure \ref{fig:sensi4}. As for the first event considered, the sensitivity functions under uniform covariation are equal for the two trees.
\end{example}

\begin{figure}
\begin{center}
\begin{subfigure}[b]{0.5\linewidth}
    \centering\includegraphics[scale=0.4]{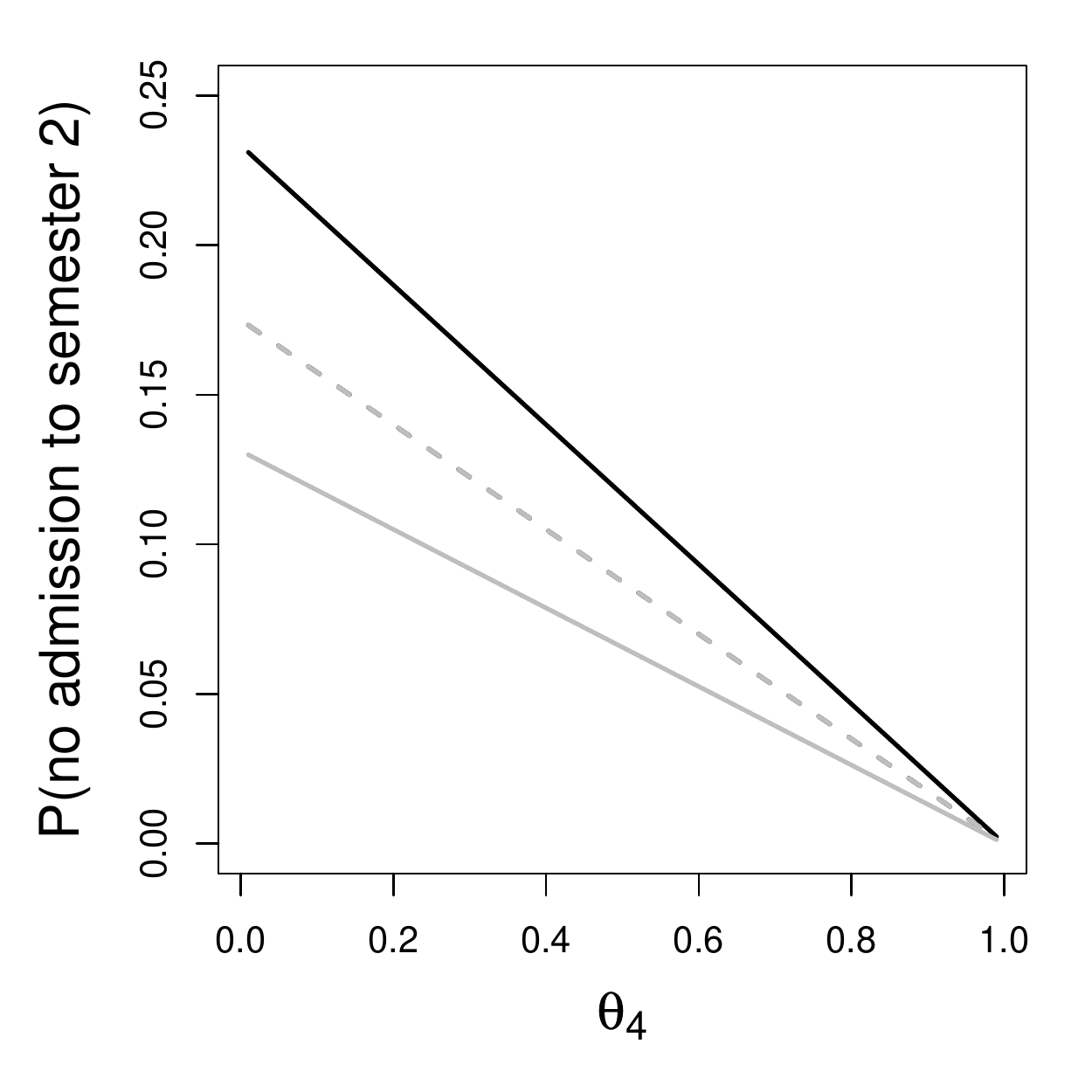}
    \caption{\label{fig:sensi1}}
  \end{subfigure}%
  \begin{subfigure}[b]{0.5\linewidth}
    \centering\includegraphics[scale=0.4]{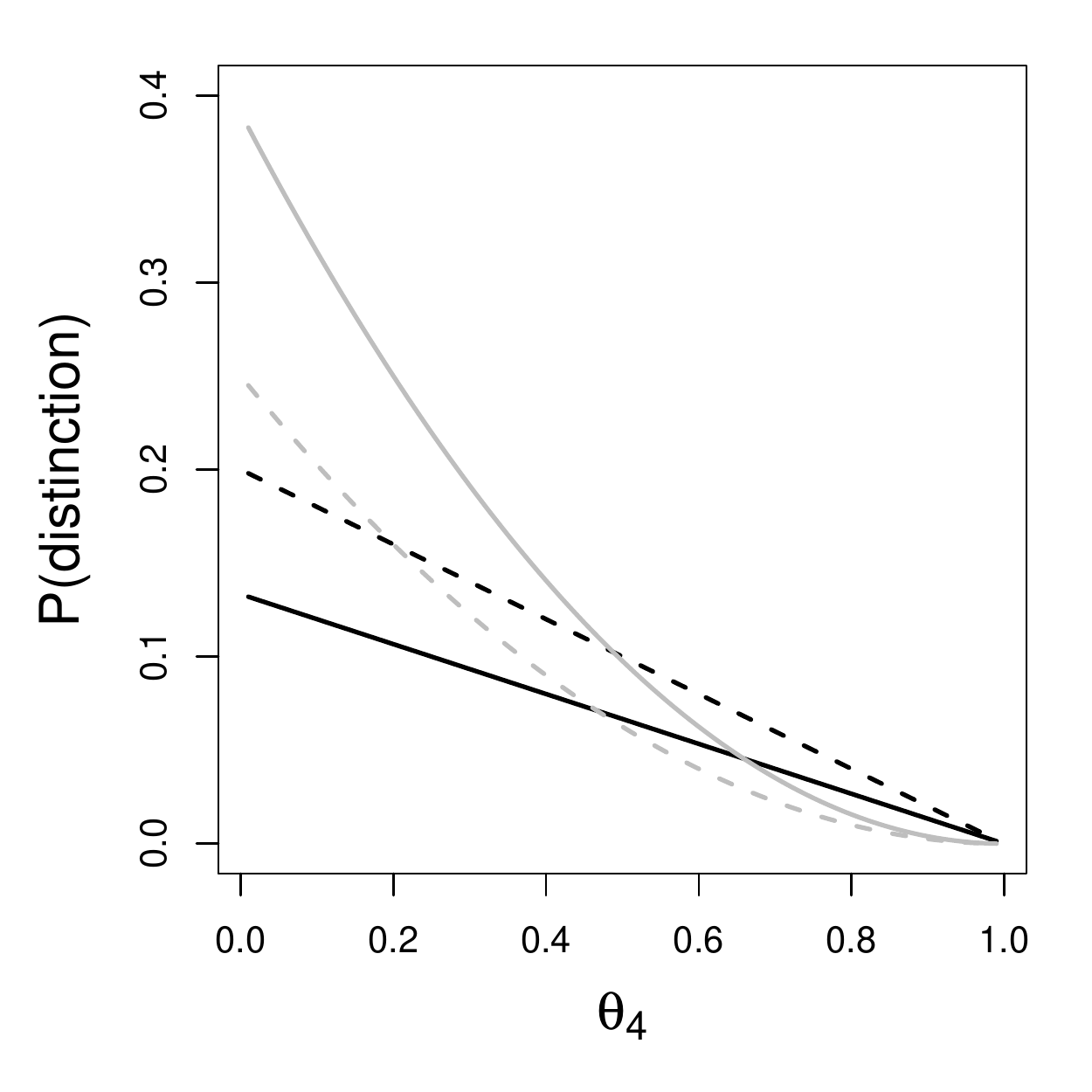}
    \caption{\label{fig:sensi2}}
  \end{subfigure}%
  \\
  \begin{subfigure}[b]{0.5\linewidth}
    \centering\includegraphics[scale=0.4]{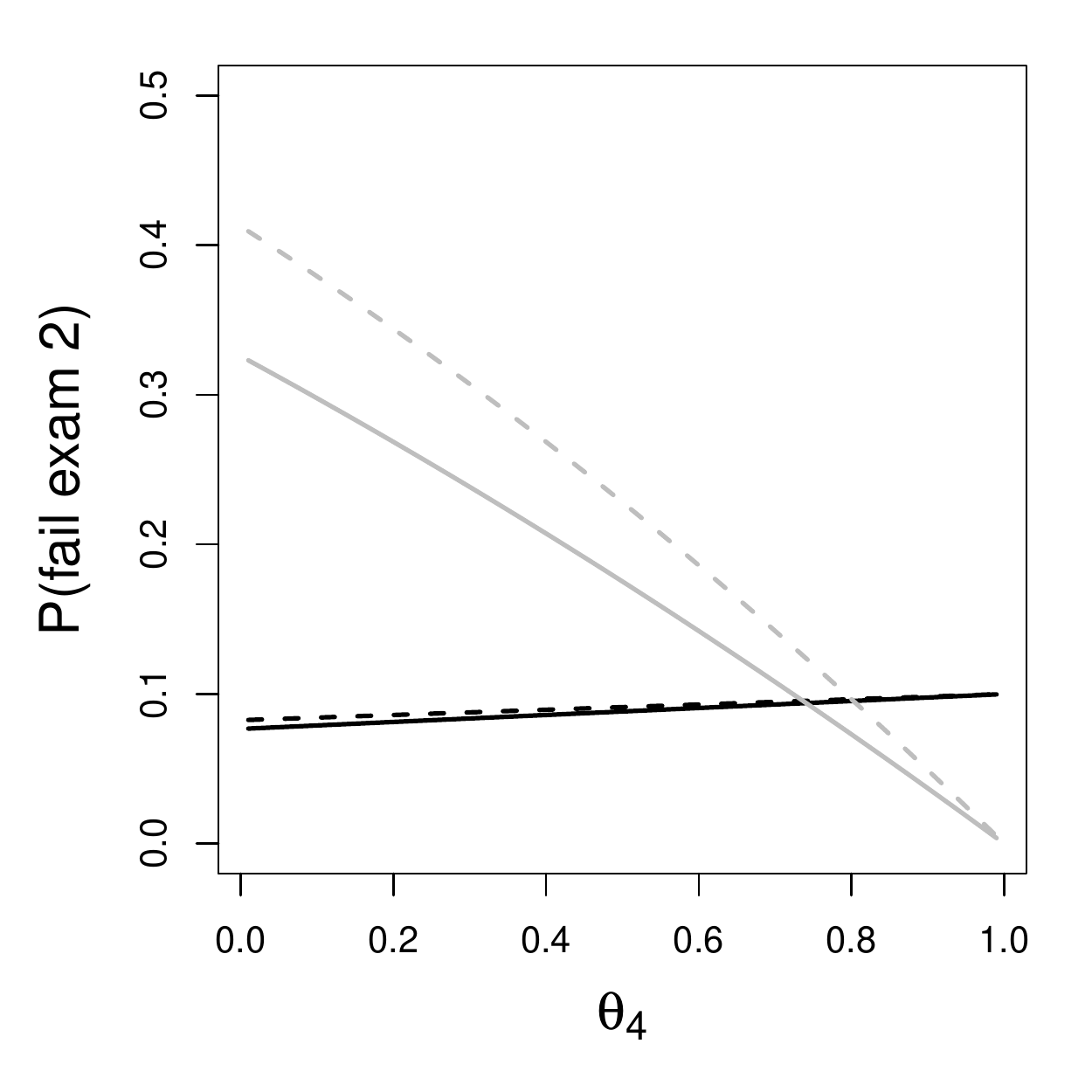}
    \caption{\label{fig:sensi3}}
  \end{subfigure}%
  \begin{subfigure}[b]{0.5\linewidth}
    \centering\includegraphics[scale=0.4]{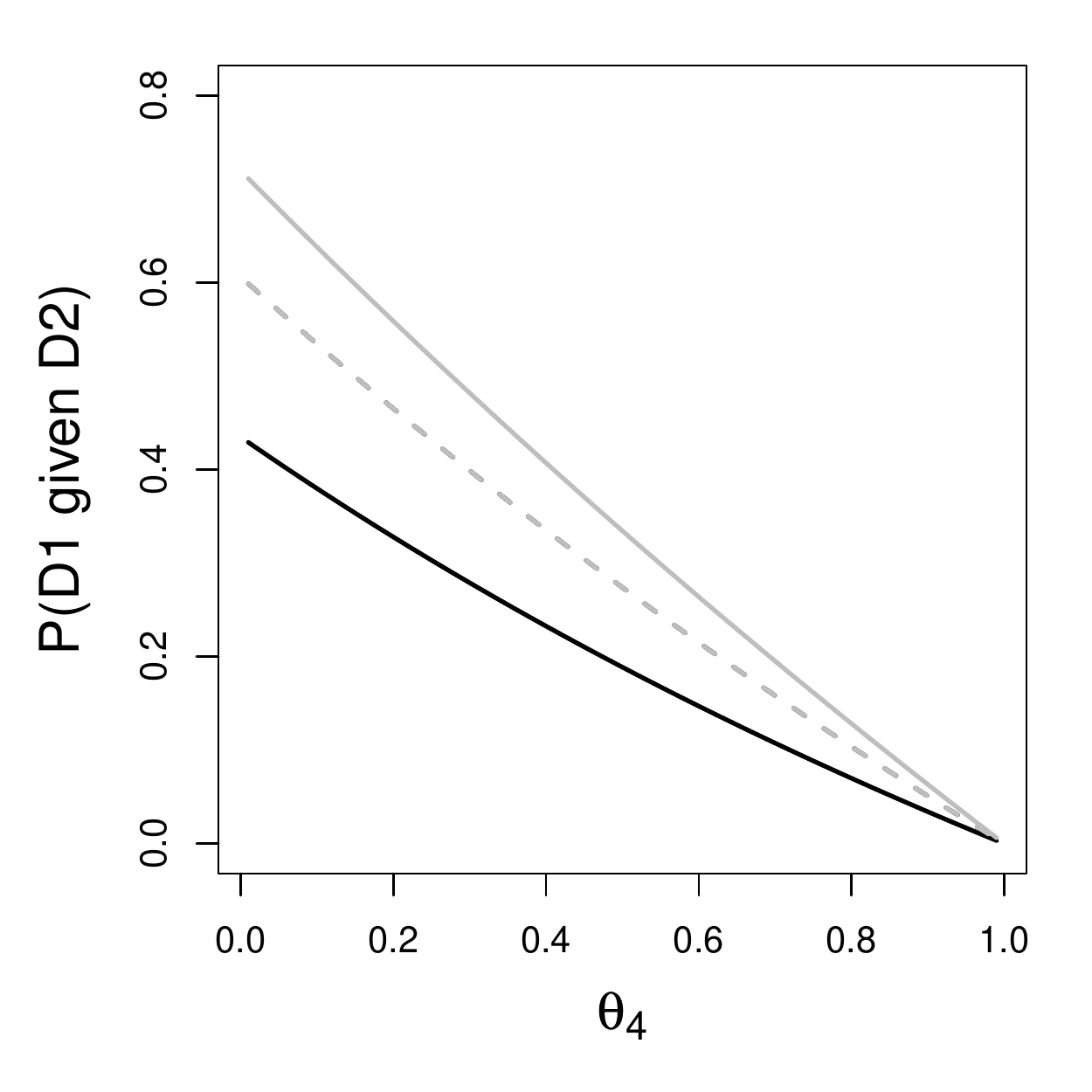}
    \caption{\label{fig:sensi4}}
  \end{subfigure}%
\end{center}
\caption{Sensitivity functions of four events for the staged trees of Section \ref{sec:23}. Black lines: multilinear staged tree; Gray lines: non-multilinear staged tree; Full lines: proportional covariation; Dashed lines: uniform covariation. \label{fig:sensi}}
\end{figure}

\section{Divergence quantification}
\label{sec:div}

Once viable parameter variations have been identified, via the study of sensitivity functions as illustrated in Section \ref{sec:sensi}, the overall effect that these would have on the model's distribution is studied. This is carried out by computing various distances and divergences between the original and the varied distributions. 

\subsection{The CD distance in non-multilinear models}
The measure of dissimilarity which is most commonly used  in sensitivity analysis in graphical models is the so-called CD distance \citep{Chan2005}.

\begin{definition}
The CD distance between two probability distributions $\tilde{\operatorname{P}}$ and $\operatorname{P}$ over a discrete sample space $\mathbb{Y}$ is 
\[
\mathcal{D}_{\operatorname{CD}}(\tilde{\operatorname{P}},\operatorname{P})=\log\max_{y\in\mathbb{Y}}\frac{\tilde{\operatorname{P}}(y)}{\operatorname{P}(y)}-\log\min_{y\in\mathbb{Y}}\frac{\tilde{\operatorname{P}}(y)}{\operatorname{P}(y)}.
\]
\end{definition}

For single and specific multi-way parameter variations, proportional covariation minimizes the CD distance in BN models, as well as in any multilinear MM \citep{Chan2002,Leonelli2017}. However, in non-multilinear models even for single parameter variations proportional covariation does not minimize the CD distance in general as shown by the following example.

\begin{example}
\label{ex:3}
Consider two random variables $Y_1$ and $Y_2$ and suppose $\mathbb{Y}_1=\mathbb{Y}_{2}=[3]$. Suppose also \[
\theta_{i}=\operatorname{P}(Y_1=i)=\operatorname{P}(Y_2=i\,|\,Y_1=j), \,\,\,\,\, i\in[3],j\in[2].
\] 
and 
$
\theta_{i+3}=\operatorname{P}(Y_2=i\,|\, Y_1=3).$ 
The atomic probabilities of this model are clearly non-multilinear. Suppose $\theta_{i}$ is varied and $\theta_{2}$ and $\theta_{3}$ are covaried. Suppose $\theta_{1}=0.33$, $\theta_{2}=0.33$, $\theta_{3}=0.34$ and let $\theta_{1}$ be varied to $0.4$ (the value of $\theta_4$, $\theta_5$ and $\theta_6$ does not affect the CD distance). In this situation the CD distance under a proportional scheme is $2.52$, whilst under a uniform scheme the distance it equals $2.50$. For this parameter variation, the uniform scheme would then be preferred to a proportional one if a user wishes to minimize the CD distance. Conversely, if $\theta_{1}$ is set to $0.2$ the distance is smaller under the proportional scheme $(2.89)$ than under the uniform one $(2.92)$.
\end{example}

Next the form of the CD distance in MMs is derived in general and for specific covariation schemes. For all $\emptyset\neq H\subset [k]$ define $\mathbb{Y}_H^{=}=\{y\in\mathbb{Y}:A_{y,i}=0 \mbox{ for all } i\in H\}$ and let $\mathbb{Y}_H^{\neq}=\mathbb{Y}\setminus \mathbb{Y}_{H}^=$. The set $\mathbb{Y}_H^{\neq}$ includes the events for which at least one parameter with index in $H$ has a non-zero exponent.

\begin{theorem}
\label{theo:2}
 Let $\operatorname{P}\in MM(A,\theta,S)$ and suppose the parameter $\theta_i$ is varied, where $i\in S_j$. Then
 \begin{itemize}
 \item for a generic $\theta_i$-covariation scheme $\sigma$
 \begin{equation}
 \label{eq:theo21}
 \mathcal{D}_{\operatorname{CD}}(\sigma(\operatorname{P}),\operatorname{P})= \log\max_{y\in\mathbb{Y}_{S_j}^{\neq}}\left(\frac{\tilde\theta_{S_j}}{\theta_{S_j}}\right)^{A_{y,S_j}} - \log\min_{y\in\mathbb{Y}_{S_j}^{\neq}}\left(\frac{\tilde\theta_{S_j}}{\theta_{S_j}}\right)^{A_{y,S_j}}
 \end{equation}
 \item for proportional covariation $\sigma_{\operatorname{pro}}$
 \[
 \mathcal{D}_{\operatorname{CD}}(\sigma_{\operatorname{pro}}(\operatorname{P}),\operatorname{P})= \log\max_{y\in\mathbb{Y}_{S_j}^{\neq}}\left(\frac{\tilde\theta_{i}}{\theta_{i}}\right)^{A_{y,i}}\left(\frac{1-\tilde\theta_i}{1-\theta_i}\right)^{|A_{y,S_j^{-i}}|} - \log\min_{y\in\mathbb{Y}_{S_j}^{\neq}}\left(\frac{\tilde\theta_{i}}{\theta_{i}}\right)^{A_{y,i}}\left(\frac{1-\tilde\theta_i}{1-\theta_i}\right)^{|A_{y,S_j^{-i}}|}
 \]
\item for uniform covariation $\sigma_{\operatorname{uni}}$
\[
 \mathcal{D}_{\operatorname{CD}}(\sigma_{\operatorname{uni}}(\operatorname{P}),\operatorname{P})= \log\max_{y\in\mathbb{Y}_{S_j}^{\neq}} \frac{\theta_i^{A_{y,i}}\left(\frac{1-\tilde\theta_i}{\# S_j -1}\right)^{|A_{y,S_j^{-1}}|}}{\theta_{S_j}^{A_{y,S_j}}}- \log\min_{y\in\mathbb{Y}_{S_j}^{\neq}}\frac{\theta_i^{A_{y,i}}\left(\frac{1-\tilde\theta_i}{\# S_j -1}\right)^{|A_{y,S_j^{-1}}|}}{\theta_{S_j}^{A_{y,S_j}}}
\]
\item for linear covariation $\sigma_{\operatorname{lin}}$
\[
 \mathcal{D}_{\operatorname{CD}}(\sigma_{\operatorname{lin}}(\operatorname{P}),\operatorname{P})= \log\max_{y\in\mathbb{Y}_{S_j}^{\neq}} \left(\frac{\tilde\theta_i}{\theta_i}\right)^{A_{y,i}}\prod_{k\in S_j}\left(\frac{\gamma_k\tilde\theta_i+\delta_k}{\theta_k}\right)^{A_{y,k}}- \log\min_{y\in\mathbb{Y}_{S_j}^{\neq}}\left(\frac{\tilde\theta_i}{\theta_i}\right)^{A_{y,i}}\prod_{k\in S_j}\left(\frac{\gamma_k\tilde\theta_i+\delta_k}{\theta_k}\right)^{A_{y,k}}
\]
 \end{itemize}
\end{theorem}

\begin{proof}
For equation (\ref{eq:theo11}) notice that 
\begin{eqnarray*}
 \mathcal{D}_{\operatorname{CD}}(\sigma(\operatorname{P}),\operatorname{P}) &=& \log\max_{y\in\mathbb{Y}}\left(\frac{\tilde\theta^{A_y}}{\theta^{A_y}}\right) - \log\min_{y\in\mathbb{Y}}\left(\frac{\tilde\theta^{A_y}}{\theta^{A_y}}\right) \\
 &=& \log\max_{y\in\mathbb{Y}}\left(\frac{\tilde\theta_{S_j}^{A_{y,S_j}}\tilde\theta_{-S_j}^{A_{y,-S_j}}}{\theta_{S_j}^{A_{y,S_j}}\theta_{-S_j}^{A_{y,-S_j}}}\right)-\log\min_{y\in\mathbb{Y}}\left(\frac{\tilde\theta_{S_j}^{A_{y,S_j}}\tilde\theta_{-S_j}^{A_{y,-S_j}}}{\theta_{S_j}^{A_{y,S_j}}\theta_{-S_j}^{A_{y,-S_j}}}\right)\\
 &=& \log\max_{y\in\mathbb{Y}}\left(\frac{\tilde\theta_{S_j}}{\theta_{S_j}}\right)^{A_{y,S_j}}-\log\min_{y\in\mathbb{Y}}\left(\frac{\tilde\theta_{S_j}}{\theta_{S_j}}\right)^{A_{y,S_j}}\\
 &=&\log\max_{y\in\mathbb{Y}_{S_j}^{\neq}}\left(\frac{\tilde\theta_{S_j}}{\theta_{S_j}}\right)^{A_{y,S_j}} - \log\min_{y\in\mathbb{Y}_{S_j}^{\neq}}\left(\frac{\tilde\theta_{S_j}}{\theta_{S_j}}\right)^{A_{y,S_j}},
\end{eqnarray*}
where the last equality holds since, for all $y\in\mathbb{Y}_{S_j}^{=}$, $(\tilde\theta_{S_j}/\theta_{S_j})^{A_{y,S_j}}=1$ and there are always both larger and smaller ratios between varied and original parameters.

The form of the CD distance under different covariation schemes follows from equation (\ref{eq:theo21}) by plugging-in their definition given in Definition \ref{def:covar}. 
\end{proof}

One of the reasons why the CD distance is commonly used for sensitivity analysis in BNs is that, for a single parameter variation, the distance between the BN distributions equals the distance between the single conditional probability distributions associated to the varied parameters \citep{Chan2002}. Theorem \ref{theo:2} demonstrates that this is true in general for non-multilinear models since the distance only depends on the parameter $\theta_{S_j}$.

\begin{example}
\label{ex:5}
As in Example \ref{ex:3}, suppose the parameter $\theta_4$ is varied in the two staged trees from the educational example of Section \ref{sec:23}. From the results of \citet{Leonelli2017}, it can be deduced that for the multilinear tree, the CD distance between the original and varied distributions is simply
\begin{equation}
\label{eq:cduni}
\log\max_{i=3,4,5}\frac{\tilde\theta_i}{\theta_i}-\log\min_{i=3,4,5}\frac{\tilde\theta_i}{\theta_i}.
\end{equation}
Conversely, using Theorem \ref{theo:2}, for the non-multilinear staged tree this equals
\begin{equation}
\label{eq:cd}
\log\max\left\{\frac{\tilde\theta_3}{\theta_3},\frac{\tilde\theta_3^2}{\theta_3^2},\frac{\tilde\theta_4^2}{\theta_4^2},\frac{\tilde\theta_5^2}{\theta_5^2},\frac{\tilde\theta_3\tilde\theta_4}{\theta_3\theta_4},\frac{\tilde\theta_3\tilde\theta_5}{\theta_3\theta_5},\frac{\tilde\theta_4\tilde\theta_5}{\theta_4\theta_5}\right\}-\log\min \left\{\frac{\tilde\theta_3}{\theta_3},\frac{\tilde\theta_3^2}{\theta_3^2},\frac{\tilde\theta_4^2}{\theta_4^2},\frac{\tilde\theta_5^2}{\theta_5^2},\frac{\tilde\theta_3\tilde\theta_4}{\theta_3\theta_4},\frac{\tilde\theta_3\tilde\theta_5}{\theta_3\theta_5},\frac{\tilde\theta_4\tilde\theta_5}{\theta_4\theta_5}\right\}.
\end{equation}
The specific form of the CD distance for uniform covariation can be deduced from equation (\ref{eq:cd}) by simply substituting $\tilde\theta_3$ and $\tilde\theta_5$ with $(1-\tilde\theta_4)/2$. For proportional covariation the CD distance greatly simplifies and can be written as
\[
\log\max\left\{\frac{\tilde\theta_4^2}{\tilde\theta_4^2},\frac{1-\tilde\theta_4}{1-\theta_4},\frac{(1-\tilde\theta_4)^2}{(1-\theta_4)^2},\frac{\tilde\theta_4(1-\tilde\theta_4)}{\theta_4(1-\theta_4)}\right\}-\log\min\left\{\frac{\tilde\theta_4^2}{\tilde\theta_4^2},\frac{1-\tilde\theta_4}{1-\theta_4},\frac{(1-\tilde\theta_4)^2}{(1-\theta_4)^2},\frac{\tilde\theta_4(1-\tilde\theta_4)}{\theta_4(1-\theta_4)}\right\},
\]
which, as formalized by Theorem \ref{theo:2}, only depends on the original and varied values of $\theta_4$. 

The CD distances for proportional and uniform covariation and any possible varied value of $\theta_4$ are reported in Figure \ref{fig:CD}. Although for the two trees the shape of the distances are similar, for the non-multilinear tree the CD distance is larger. Notice that although for this application the CD distance for proportional covariation is always smaller than for uniform covariation, Example \ref{ex:3} above gives an illustration where this is not the case.
\end{example}

\begin{figure}
\begin{center}
\begin{subfigure}[b]{0.5\linewidth}
    \centering\includegraphics[scale=0.4]{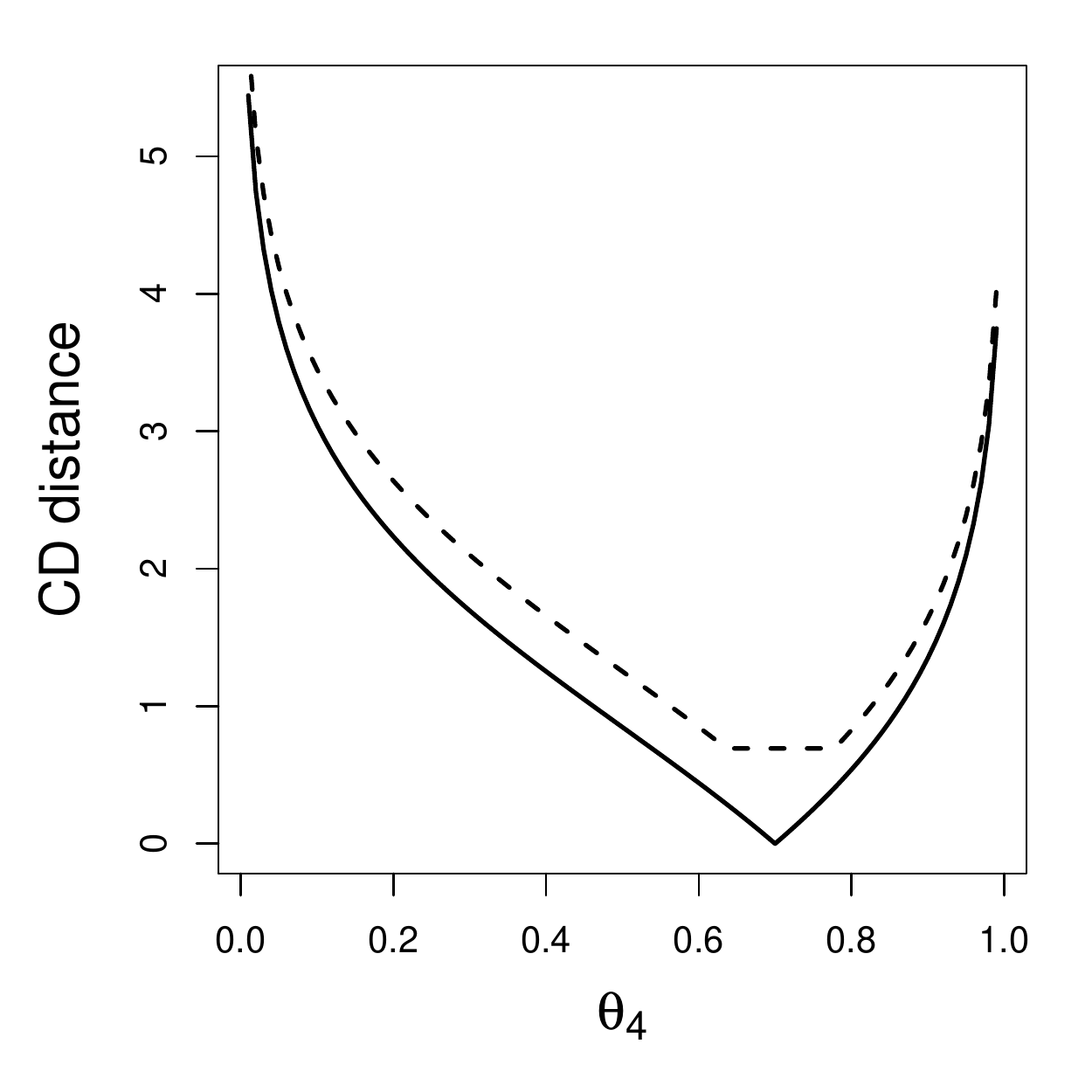}
    \caption{Multilinear tree.\label{fig:CD1}}
  \end{subfigure}%
  \begin{subfigure}[b]{0.5\linewidth}
    \centering\includegraphics[scale=0.4]{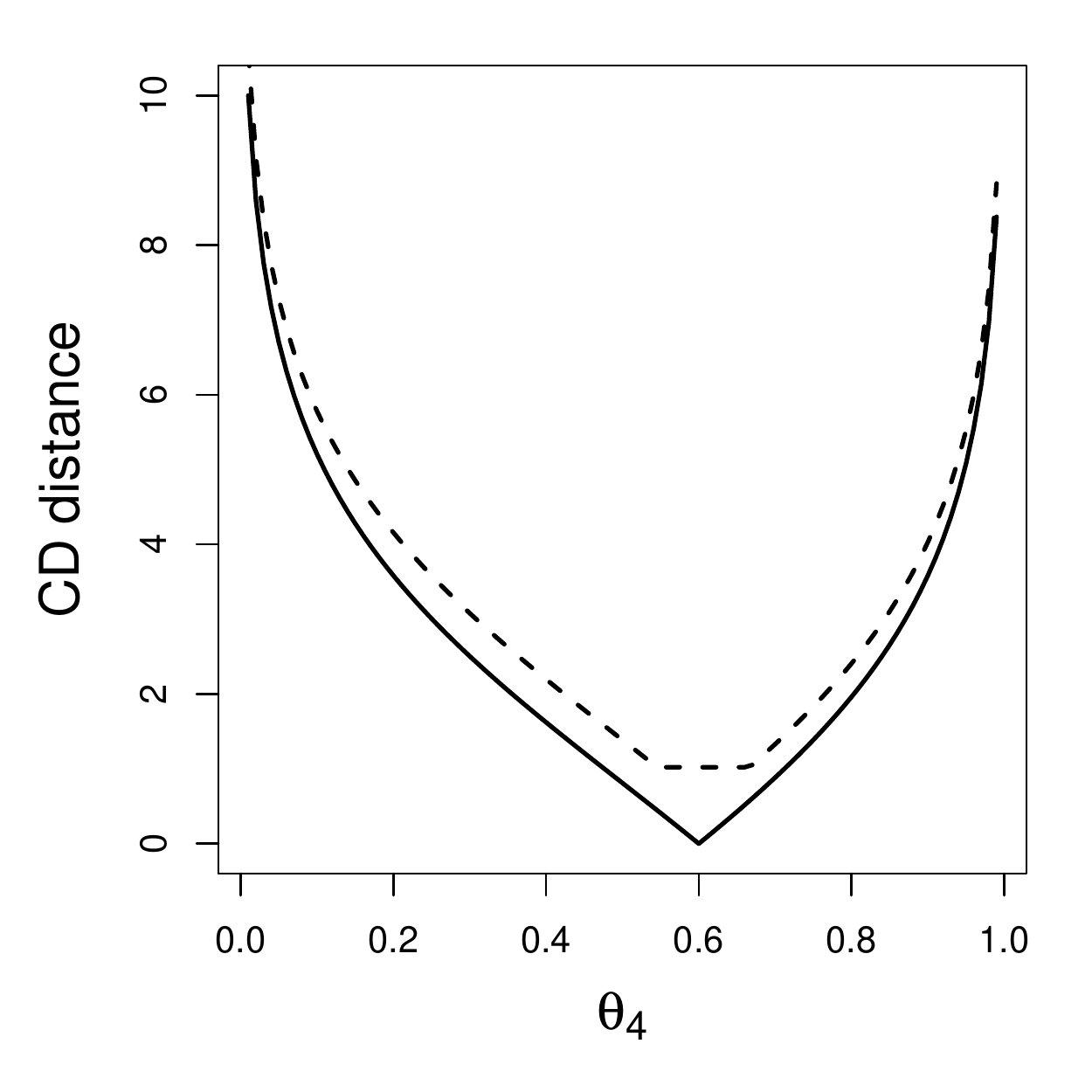}
    \caption{Non-multilinear tree.\label{fig:CD2}}
  \end{subfigure}%
  \end{center}
\caption{CD distance for the staged trees of Section \ref{sec:23} for variations of  $\theta_4$. Full lines: proportional covariation; Dashed lines: uniform covariation. \label{fig:CD}}
\end{figure}

Theorem \ref{theo:2} and Example \ref{ex:5} show that for single parameter variations the CD distance in non-multilinear models does not simply correspond to the distance between distributions defined over one element of the partition $S$ (as in equation (\ref{eq:cduni}) for the multilinear staged tree). However, there are parameter variations in non-multilinear model where this is the case as formalized by Corollary \ref{cor:4}

\begin{corollary}
\label{cor:4}
In the notation of Theorem \ref{theo:2}, suppose $0\leq |A_{y,S_j}|\leq 1$ for all $y\in\mathbb{Y}_{S_j}^{\neq}$. Then
\begin{itemize}
\item for a generic $\theta_i$-covariation scheme $\sigma$
 \begin{equation}
 \label{eq}
 \mathcal{D}_{\operatorname{CD}}(\sigma(\operatorname{P}),\operatorname{P})= \log\max_{i\in S_j}\frac{\tilde\theta_i}{\theta_i}-\log\min_{i\in S_j}\frac{\tilde\theta_i}{\theta_i}
 \end{equation}
 \item for proportional covariation $\sigma_{\operatorname{pro}}$
 \[
 \mathcal{D}_{\operatorname{CD}}(\sigma_{\operatorname{pro}}(\operatorname{P}),\operatorname{P})= \left|\log\frac{\tilde\theta_i}{\theta_i}-\log\frac{1-\tilde\theta_i}{1-\theta_i}\right|
 \]
\item for uniform covariation $\sigma_{\operatorname{uni}}$
\[
 \mathcal{D}_{\operatorname{CD}}(\sigma_{\operatorname{uni}}(\operatorname{P}),\operatorname{P})= \log\max \left\{\frac{\tilde\theta_i}{\theta_i},\frac{1-\tilde\theta_i}{(\# S_j-1)\min_{k\in S_j^{-i}}\theta_k}\right\}-\log\min \left\{\frac{\tilde\theta_i}{\theta_i},\frac{1-\tilde\theta_i}{(\# S_j-1)\max_{k\in S_j^{-i}}\theta_k}\right\}
\]
\item for linear covariation $\sigma_{\operatorname{lin}}$, where $\delta_k=-\gamma_k$ for all $k\in S_j^{-i}$,
\[
 \mathcal{D}_{\operatorname{CD}}(\sigma_{\operatorname{lin}}(\operatorname{P}),\operatorname{P})= \log\max \left\{\frac{\tilde\theta_i}{\theta_i},\frac{1-\tilde\theta_i}{\min_{k\in S_j^{-i}}\delta_k^{-1}\theta_k}\right\}-\log\min \left\{\frac{\tilde\theta_i}{\theta_i},\frac{1-\tilde\theta_i}{\max_{k\in S_j^{-i}}\delta_k^{-1}\theta_k}\right\}
\]
\end{itemize}
\end{corollary}

\begin{proof}
Equation (\ref{eq}) follows from equation (\ref{eq:theo21}) by imposing the condition $0\leq |A_{y,S_j}|\leq 1$. Equation (\ref{eq}) then coincides to the CD distance between one conditional probability distribution in BNs and its varied version and the specific form of the distance under different covariation schemes can be derived as in \citet{Renooij2014}.
\end{proof}

Corollary \ref{cor:4} generalizes the results of \citet{Renooij2014}, which derive the specific form of the sensitivity function for various covariation schemes in BNs, to the case of non-multilinear models for specific choices of varied parameter. Importantly, the form of the CD distance derived in Corollary \ref{cor:4} has the very important consequence that for some varied parameters proportional variation can be shown to be optimal.

\begin{theorem}
\label{theo:3}
Under the conditions of Corollary \ref{cor:4}, proportional covariation minimizes the CD distance between the original and varied distribution amongst all possible covariation schemes.
\end{theorem}

\begin{proof}
The theorem follows from equation (\ref{eq}) which is the CD distance between one conditional probability distribution in BNs and its varied version. As proven in \citet{Chan2002} this distance is minimized by proportional covariation.
\end{proof}

Theorem \ref{theo:3} therefore extends the results of \citet{Chan2002} and \citet{Leonelli2017} which prove the optimality of proportional covariation for BNs and multilinear MMs to specific sensitivity analyses in non-multilinear models.

\begin{example}
For the non-multilinear staged tree in Figure \ref{fig:tree3}, consider the stage $\{F_{1,A},F_{1,B}\}$. Suppose there is an additional edge coming out of this stage ending in a leaf (for example by splitting the fail result, into badly failed and moderately fail). Then one could show that the columns associated to the parameters of the stage probability distribution in the $A$ matrix have only zero or one entries. This can also be seen graphically since $F_{1,A}$ and $F_{1,B}$ are not along a same root-to-leaf path. Therefore, by Theorem \ref{theo:3}, if one probability from this stage distribution is varied then by proportionally covarying the remaining parameters the CD distance between the original staged tree distribution and the new one is minimized.
\end{example}

\subsection{$\phi$-divergences in non-multilinear models}

Another class of divergences which is often used in practice is the so-called $\phi$-divergence \citep{Ali1966}.

\begin{definition}
The $\phi$-divergence from $\tilde{\operatorname{P}}$ to $\operatorname{P}$ over a discrete sample space $\mathbb{Y}$ is 
\[
\mathcal{D}_\phi(\tilde{\operatorname{P}},\operatorname{P})=\sum_{y\in\mathbb{Y}}\operatorname{P}(y)\phi\left(\frac{\tilde{\operatorname{P}}(y)}{\operatorname{P}(y)}\right), \hspace{1cm} \phi\in\Phi,
\]
where $\Phi$ is the class of convex functions $\phi(x)$, $x\geq 0$, such that $\phi(1)=0$, $0\phi(0/0)=0$ and $0\phi(x/0)=\lim_{x\rightarrow +\infty}\phi(x)/x$.
\end{definition}

By definition, and conversely to CD distances, $\phi$-divergences are not symmetric, i.e. $\mathcal{D}_\phi(\tilde{\operatorname{P}},\operatorname{P})\neq \mathcal{D}_\phi(\operatorname{P},\tilde{\operatorname{P}})$. Notice that this class includes a large number of commonly used divergences, most notably Kullback-Leibler divergence \citep{Kullback1951} for $\phi(x)=x\log(x)$ and the inverse Kullback-Leibler divergence for $\phi(x)=-\log(x)$.

\begin{proposition}
Let $\operatorname{P}\in MM(A,\theta,S)$ and suppose the parameter $\theta_i$ is varied, where $i\in S_j$. Then for a generic $\theta_i$-covariation scheme $\sigma$
\begin{equation}
\label{eq:phi}
\mathcal{D}_\phi(\sigma(\operatorname{P}),\operatorname{P})=\sum_{y\in\mathbb{Y}_{S_j}^{\neq}}\theta^{A_y}\phi\left(\frac{\tilde\theta_{S_j}^{A_{y,S_j}}}{\theta_{S_j}^{A_{y,S_j}}}\right).
\end{equation}
\end{proposition}

\begin{proof}
Notice that
\[
\mathcal{D}_\phi(\sigma(\operatorname{P}),\operatorname{P})=\sum_{y\in\mathbb{Y}}\theta^{A_y}\phi\left(\frac{\tilde\theta^{A_y}}{\tilde\theta^{A_y}}\right)=\sum_{y\in\mathbb{Y}}\theta^{A_y}\phi\left(\frac{\tilde\theta_{S_j}^{A_{y,S_j}}}{\tilde\theta_{S_j}^{A_{y,S_j}}}\right)=\sum_{y\in\mathbb{Y}_{S_j}^{\neq}}\theta^{A_y}\phi\left(\frac{\tilde\theta_{S_j}^{A_{y,S_j}}}{\theta_{S_j}^{A_{y,S_j}}}\right)
\]
where the last equality follows by noting that for all $y\in\mathbb{Y}_{S_j}^=$ the term in the summation is $0\phi(0/0)$ which by definition is equal to zero.
\end{proof}

Notice that as for BNs and multilinear MMs, $\phi$-divergences do not depend on the parameter vector $\theta_{S_j}$ of the varied parameter only, but on the full $\theta$. Therefore, their computation in practice is more expensive than for CD distances. Furthermore, due to this extra complexity, $\phi$-divergences do not simplify greatly for specific covariation schemes. To see this, the $\phi$-divergence under proportional covariation can be written as
\[
\mathcal{D}_\phi(\sigma_{\textnormal{pro}}(\operatorname{P}),\operatorname{P})=\sum_{y\in\mathbb{Y}_{S_j}^{\neq}}\theta^{A_y}\phi\left(\left(\frac{\tilde\theta_i}{\theta_i}\right)^{A_{y,i}}\left(\frac{1-\tilde\theta_i}{1-\theta_i}\right)^{|A_{y,S_j^{-i}}|}\right),
\]
which still depends on the full parameter vector $\theta$. The specific form of the $\phi$-divergence under other covariation schemes can be easily deduced by plugging-in their definition into equation (\ref{eq:phi}).

\begin{example}
The Kullback-Leibler divergences for proportional and uniform covariation and any possible varied value of $\theta_4$ in the trees of Section \ref{sec:23} are reported in Figure \ref{fig:CD}. The form and the value of the divergences for the two trees are similar. Notice that for this example the Kullback-Leibler divergence is always smaller for proportional covariation than uniform covariation, although there is no theoretical guarantee that this is always the case.
\end{example}

\begin{figure}
\begin{center}
\begin{subfigure}[b]{0.5\linewidth}
    \centering\includegraphics[scale=0.4]{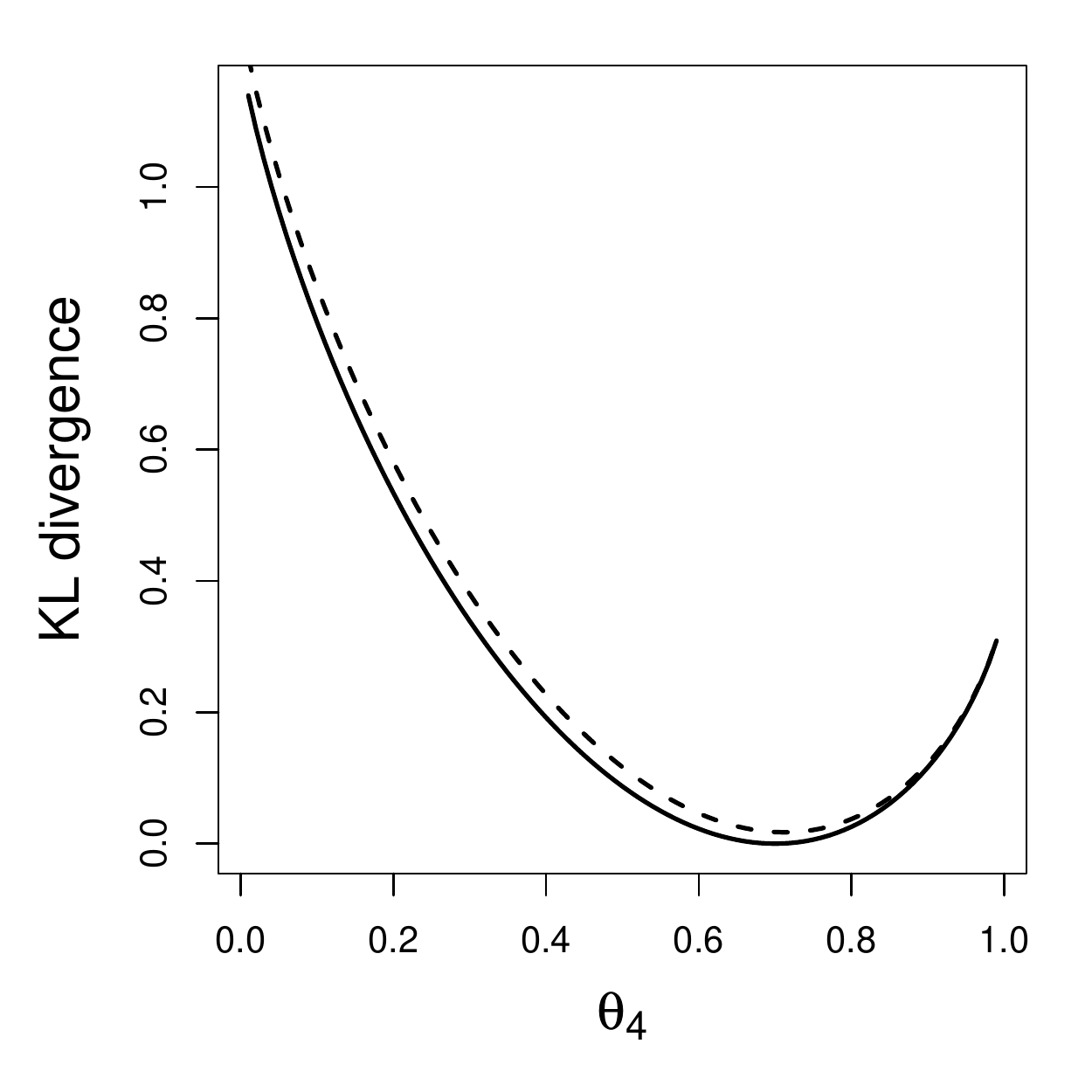}
    \caption{Multilinear tree.\label{fig:KL1}}
  \end{subfigure}%
  \begin{subfigure}[b]{0.5\linewidth}
    \centering\includegraphics[scale=0.4]{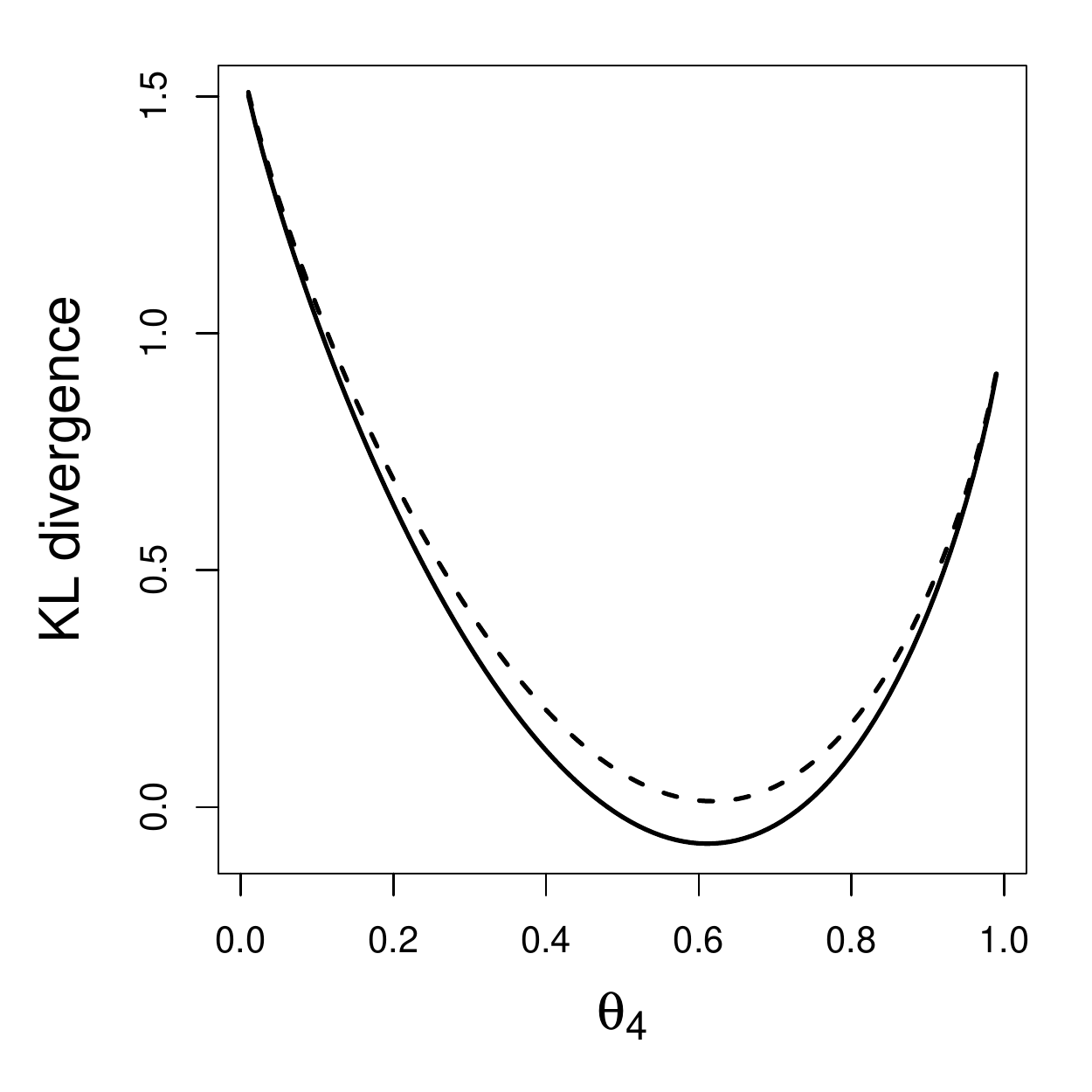}
    \caption{Non-multilinear tree.\label{fig:KL2}}
  \end{subfigure}%
  \end{center}
\caption{KL divergence for the staged trees of Section \ref{sec:23} for variations of  $\theta_4$. Full lines: proportional covariation; Dashed lines: uniform covariation. \label{fig:KL}}
\end{figure}

\section{Discussion}

The representation of probabilistic graphical models in terms of the defining atomic monomial probabilities has proven useful in sensitivity analysis. Here a general approach for this type of analyses in models whose atomic probabilities are non-multilinear, including DBNs, hidden Markov models and staged trees, is introduced. The form of the sensitivity functions and various distances/divergences is derived here for a variety of covariation schemes, and their properties studied. In general these are different to their counterparts in multilinear MMs and exhibit a more complex structure. One optimality result for proportional covariation is also presented, giving an even stronger justification for the use of this scheme in practice.

The examples presented suggest that proportional covariation minimizes both CD distances and $\phi$-divergences under much milder conditions than the ones given in Theorem \ref{theo:3}. However, it is currently unknown under which conditions proportional covariation is optimal in general. General conditions of optimality in multilinear models have been derived only recently in \citet{Leonelli2018}. The identification of these in the more general case of non-multilinear case is the subject of ongoing research.

Software for carrying out sensitivity analysis in practice is still very limited \citep[see][for a notable exception]{samIam}. A package for sensititivity analysis in BNs, and more generally for MMs, in the open-source R software \citep{R} is currently under development. The development of such a package is critical and could be of great benefit for the whole AI community.

\section*{Acknowledgements}
The author kindly thanks Christiane G\"{o}rgen and Jim Q. Smith for comments on previous versions of the manuscript.

\section*{References}
\bibliographystyle{plainnat} 
\bibliography{bib1}

\end{document}